\documentclass[a4paper,11pt]{amsart}
\usepackage{amsfonts,amsthm,amsmath,amssymb,graphicx}
\usepackage{fullpage}
\usepackage{tikz}

\title{Hausdorff dimension in inhomogeneous Diophantine approximation}            

\date{\today}

\author{Yann Bugeaud}
\address{
 IRMA, UMR 7501,
 Universit\'e de Strasbourg, CNRS,
 7, rue Ren\'e Descartes, 67000 Strasbourg, France.
{\it bugeaud@math.unistra.fr}}

\author{Dong Han Kim}
\address{
 Department of Mathematics Education,
Dongguk University -- Seoul,
30 Pildong-ro 1-gil, Jung-gu, Seoul, 04620 Korea. 
{\it kim2010@dongguk.edu}}

\author{Seonhee Lim}
\address{
 Department of Mathematical Sciences and Research Institute of Mathematics,
Seoul National University,
Kwanak-ro, Kwanak-gu, Seoul, Korea.
{\it slim@snu.ac.kr}}

\author{Micha{\l}  Rams}
\address{
 Institute of Mathematics,
Polish Academy of Sciences,
ul. \'Sniadeckich 8, 00-656 Warszawa, Poland.
{\it rams@impan.gov.pl}}

\theoremstyle{plain}
\newtheorem{lem}{Lemma}[section]
\newtheorem{prop}[lem]{Proposition}
\newtheorem{thm}[lem]{Theorem}

\theoremstyle{definition}
\newtheorem{defn}[lem]{Definition}

\theoremstyle{remark}
\newtheorem*{rem}{Remark}

\newtheorem*{que}{Question}
\newtheorem*{obs}{Observation}
\newtheorem*{pro}{Problem}

\numberwithin{equation}{section}

\newcommand{\cJ}{\mathcal J}

\newcommand{\R}{\mathbb R}
\newcommand{\Z}{\mathbb Z}

\renewcommand{\epsilon}{\varepsilon}

\newcommand{\ep}{\epsilon}

\newcommand{\Bad}{\mathrm{Bad}}

\def\cA{{\mathcal P}}

\def\cS{{\mathcal S}}

\def\bfy{{\bf y}}
\def\bfx{{\bf x}}
\def\uz{{\mathbf z}}
\def\eps{{\varepsilon}}

\def\bR{{\mathbb R}}
\def\bZ{{\mathbb Z}}
\def\utheta{{\underline{\theta}}}
 
\def\ux{{\mathbf{x}}} 
\def\uy{{\mathbf{y}}}
\def\uq{{\mathbf{q}}}

\def\omc{{\hat{\omega}}}
\def\Card{{\rm Card}}

\begin{document}

\maketitle

\def\thefootnote{}
\def\thefootnote{\arabic{footnote}}
\setcounter{footnote}{0}

\begin{abstract}
Let $\alpha$ be an irrational real number.  
We show that the set of $\epsilon$-badly approximable numbers
\[
\mathrm{Bad}^\varepsilon (\alpha) := \{x\in [0,1]\, : \, \liminf_{|q| \to \infty} |q| \cdot \| q\alpha -x \| \geq \varepsilon \}
\]
has full Hausdorff dimension for some positive $\epsilon$ if and only if $\alpha$ is singular on average. 
The condition is equivalent to the average $\frac{1}{k} \sum_{i=1, \cdots, k} \log a_i$ of the logarithms   
of the partial quotients $a_i$ of $\alpha$ going to infinity with $k$. 
We also consider one-sided approximation, 
obtain a stronger result when $a_i$ tends to infinity, and 
establish a partial result in higher dimensions. 
\end{abstract}

\section{Introduction and results}    

A well-known result of Minkowski \cite{Min01}
asserts that, for every irrational real number $\alpha$ and every real number $x$, 
which is not of the form $m \alpha + n$ for integers $m,n$, 
there are infinitely many integers $q$ with
$$
|q| \cdot \| q \alpha - x \| < {1 \over 4},
$$
where $\| z \|$ denotes the distance from $z$ to its nearest integer.     
This result was complemented by Kim \cite{Kim07} 
who proved that, for every irrational real number $\alpha$, 
the set  of real numbers $x$ in $[0,1]$ such that 
$
 \liminf_{q \rightarrow \infty} q \cdot \|q\alpha-x \| = 0
$
has full Lebesgue measure. Subsequently, it has been proved in \cite{BHKV} 
(see also \cite{Tseng09,Mosh11}) that the complement set
is large, namely, 
for every irrational real number $\alpha$, the set 
$$
\Bad (\alpha) := \{x\in [0,1]\, : \, \liminf_{|q| \rightarrow \infty} |q| \cdot \|q\alpha-x \| > 0 \} 
$$
has Hausdorff dimension 1. 
These results seem to indicate that all 
irrational real numbers behave in a same way, independently of  
their Diophantine properties. 

However, if we refine our question and ask whether, 
for some positive real number $\ep$, the set  

\[
\Bad^\ep (\alpha) := \{x\in [0,1]\, : \, \liminf_{|q| \to \infty} |q| \cdot \| q\alpha -x \| \geq \epsilon\} 
\]
is also of Hausdorff dimension 1, then we can distinguish
$\alpha$'s of distinct Diophantine properties. This is precisely the theme of this article.

\subsection{Real numbers}

It has been proved recently in \cite{LSS} that, for almost every $\alpha$, we have 
\begin{equation} \label{eqn:0.1}
\forall \ep>0, \,\,
\dim_H \Bad^\ep (\alpha)<1.\end{equation} 
In the same article, a sufficient condition which ensures  \eqref{eqn:0.1}, 
called heaviness (see \eqref{eqn:1.2} below), was given. 

Our first main result gives necessary and sufficient conditions for \eqref{eqn:0.1} in two directions,      
one in terms of the convergents of $\alpha$ and the other in terms of singularity of $\alpha$.

Throughout the article, for 
$\alpha=  [a_0; a_1, a_2, \cdots] := a_0 + \frac{1}{a_1 + \frac{1}{a_2 + \frac{1}{a_3 + \cdots}} }$, its $k$-th convergent is denoted by 
$$
\frac{p_k}{q_k} = [a_0; a_1, \cdots, a_k] 
=a_0 + \frac{1}{a_1 + \frac{1}{a_2 + \frac{1}{a_3 + \cdots+\frac{1}{a_k}}} }.
$$
We call $(a_k)_{k \ge 0}$ the \emph{partial quotients} of $\alpha$.
An irrational real number $\alpha$ is called \emph{singular on average} if, for every $c>0$,
$$ 
\lim_{N \to \infty} \frac{1}{N} \, \Card \{ \ell \in \{1, \cdots, N\} 
:  \| q \alpha \| \le c 2^{-\ell} \mathrm{\;has \;a \;solution \; with \;} 0 < q \le 2^{\ell} \} =1.     
$$
We establish in Section 4 that this property is equivalent to the fact that the sequence 
$(q_k^{1/k})_{k \ge 1}$ tends to infinity, stated in the following theorem.     

\begin{thm} \label{thm:general}
Let $\alpha$ be an irrational real number and, for $k \ge 1$, let  
$q_k$ denote the denominator of its $k$-th convergent. 
Then the following are equivalent.
\begin{enumerate}
\item[(i)]
For some $\epsilon>0$, the set $\Bad^\eps (\alpha)$ has full Hausdorff dimension.       
\item[(ii)]
$\lim_{k \to \infty} q_k^{1/k}=\infty.$
\item[(iii)] 
$\alpha$ is singular on average.
\end{enumerate}
Moreover, if $\dim_H \Bad^\eps (\alpha) =1$ for some 
$\ep>0$, then so it is for every $\ep$ in $(0, 2^{-4} \cdot 3^{-3})$.      
\end{thm}

The equivalence between Conditions (i) and (ii) is proved in Section 2. 

According to \cite{LSS}, an irrational real number $\alpha = [a_0 ; a_1, a_2, \ldots]$  is called 
{\it heavy} if, for every $\delta > 0$, there exists $\eta > 0$ such that
\begin{equation}\label{eqn:1.2}
\liminf_{N \to \infty} \, \frac{1}{N} \sum_{k=1}^N \max\{ \log \eta a_k, 0\} \le \delta.
\end{equation}
It was shown in \cite{LSS} that if $\alpha$ is heavy, 
then $\dim_H \Bad^\ep (\alpha)<1$ for all $\ep>0$. 
Theorem \ref{thm:general} shows that the converse does not hold. 
Indeed, consider $\alpha = [0 ; a_1, a_2, \ldots]$ whose continued fraction expansion 
is defined by $a_n = 1$, for $n$ not being an integer power of $2$,    
and by $a_n = 3^n$ otherwise. 
Then, we observe that 
$(q_k^{1/k})_{k \ge 1}$ is bounded while $\alpha$ is not heavy.

Condition (ii) of Theorem \ref{thm:general} is satisfied 
when the partial quotients $(a_i)$ of
$\alpha$ tends to infinity, in which case we can slightly strengthen 
the last statement of Theorem \ref{thm:general}. 

\begin{thm} \label{thm:2s}
Let $\alpha$ be an irrational real number whose sequence of   
partial quotients tends to infinity.       
Then, for every $\ep < 1/16$, we have $\dim_H \Bad^\ep (\alpha)=1$.
\end{thm}

Theorem \ref{thm:2s} is proved in Section 2. 

An analogue for one-sided approximation   
of Minkowski's result mentioned at the beginning of the introduction was   
obtained by Khintchine \cite{Kh35}, who established that, 
for every irrational real number $\alpha$, every real number $x$, and every positive $\ep$, 
there are infinitely many positive integers $q$ with
$$
q \cdot \| q \alpha - x \| < {1 + \ep \over \sqrt{5}}. 
$$
This statement motivates the study of the set
\[
\Bad^{\ep}_+ (\alpha) := \{x\in [0,1] \, : \, \liminf_{q \to + \infty} q \cdot \| q \alpha -x \| \geq \ep\}.
\]
Our main result in this direction is the following, more precise, theorem.

\begin{thm} \label{thm:14}
Let $\alpha$ be an irrational whose sequence of partial quotients    
tends to infinity. Then, for every $\ep < 1/4$, we have
$$
\dim_H  \{x\in [0,1] \, ; \, \liminf_{q \to \infty} q \cdot \| q \alpha -x \| = \ep\} =1,
$$
while, for any $\ep>1/4$, we have
$$
 \{x\in [0,1] \, ; \, \liminf_{q \to \infty} q \cdot \| q \alpha -x \| = \ep\}=\emptyset.
$$
\end{thm}

It follows that 
$
\dim_H \Bad^\ep_+  (\alpha)=1,
$ for $\epsilon < 1/4$,
while
$
\Bad^\ep_+ (\alpha) 
$ is empty for any $\ep>1/4$.  
It still remains to determine the Hausdorff dimension of $\Bad^{1/4}_+ (\alpha)$. 
Theorem \ref{thm:14} is proved in Section 3. 

The proof that (ii) implies (i) in Theorem \ref{thm:general} rests on arguments 
already present in \cite{Cas,BHKV,BuLa05b}, which extend to Diophantine 
approximation of matrices. We discuss this more general question in the following subsection.

\subsection{Real matrices}    

If  $\ux$ is a  (column) vector  in $\R^n$, we denote by $\vert
\ux \vert$
the  maximum of the   absolute values of its coordinates. Define
$$
\Vert \ux \Vert = \min_{\uz \in \Z^n} \vert \ux -\uz \vert.
$$

Fix $m, n$ in $\mathbb{N}$ and let $A$ be an $n \times m$ real matrix. 
For $\eps>0$, we define the set
$$
\Bad^{\eps} (A)  
:= \{\ux \in [0, 1]^n :  \liminf_{\uq \in \Z^m}   |\uq|^{m/n} \cdot \| A \uq - \ux \| \ge \eps \} 
$$
and we put
$$ 
\Bad(A) := \underset{\eps>0}{\bigcup} \, \Bad^\eps (A)
= \{\ux \in [0, 1]^n :  \liminf_{\uq \in \Z^m}   |\uq|^{m/n} \cdot \| A \uq - \ux \|  > 0  \},  
$$
$$
\Bad^{\infty}(A) := \underset{\eps>0}{\bigcap} \, \Bad^\eps (A)
 = \{\ux \in [0, 1]^n :   \liminf_{\uq \in \Z^m}   |\uq|^{m/n} \cdot  \| A \uq - \ux \| = + \infty \}.  
$$
Theorem 1 of  \cite{BHKV} asserts that 
\begin{equation}\label{eqn:2.0}
\dim_H \Bad  (A) = n.     
\end{equation}

Before stating our main result in higher dimension, let us introduce some definitions and explain the 
general principle behind the proof of \eqref{eqn:2.0}. 
Dirichlet's Theorem implies that, for any $X > 1$,
the inequalities
$$
\Vert A\ux \Vert \le  X^{-m/n}
\quad {\rm and} \quad 0 < \vert \ux \vert \le X  
$$
have a solution  $\ux$ in $\Z^m$. 
The following definition 
of singularity goes back to
Khintchine \cite{KhB}.

\begin{defn}
Let $m, n$ be positive integers and $A$ a $n \times m$ real matrix. 
\begin{enumerate}
\item
The matrix $A$ is called \emph{singular} if, for every $c>0$, the 
inequalities \eqref{eqn:1.1}
\begin{equation}\label{eqn:1.1}
\Vert A\ux \Vert \le c \, X^{-m/n}
\quad {\rm and} \quad 0 < \vert \ux \vert \le X
\end{equation}

 have a solution  $\ux$ in $\Z^m$ for any sufficiently large $X$.
\item The matrix $A$ is called \emph{singular on average} if, for every $c>0$,
$$ 
\lim_{N \to \infty} \frac{1}{N} \, \Card \{ \ell \in \{1, \cdots, N\} 
: \mathrm{\; the \; inequalities \;} \eqref{eqn:1.1} \mathrm{\;have \;a \;solution \;for} \; X=2^\ell \} =1.
$$
\item
The matrix $A$ is called \emph{very well uniformly approximable} if there exists
a positive $\eps$ such that the 
inequalities 
\begin{equation} 
\Vert A\ux \Vert \le  X^{- \eps - m/n}
\quad {\rm and} \quad 0 < \vert \ux \vert \le X
\end{equation}
have a solution  $\ux$ in
$\Z^m$ for any sufficiently large $X$.
\end{enumerate}

\end{defn}
If $n=m=1$ and $A=(\alpha)$, then 
we say that $\alpha$ is
\emph{singular
(resp., singular on average, 
very well uniformly approximable) }
if $(\alpha)$ has this property.

\begin{rem}
As far as we are aware, the notion of \emph{singular on average matrices} 
has been introduced in \cite{KKLM}, motivated by the dynamical notion 
of points which escape on average under the action of a semigroup. 
The terminology \emph{very well uniformly approximable} refers to the hat exponents 
introduced in \cite{BuLa05b}. 
\end{rem}

\begin{rem}If  the subgroup $G_A= A \bZ^m + \bZ^n$ of $\bR^n$ 
has rank $\mathrm{rk}_\mathbb{Z} (G_A)$ smaller than $m+n$, 
then there exists arbitrarly large $\ux$ in $\Z^m$ 
such that $\Vert A \ux \Vert = 0.$ 
Throughout the paper, we 
consider only matrices $A$ for which $\mathrm{rk}_\mathbb{Z}(G_A)=m+n$. 
\end{rem}

When $m=n=1$, 
using the theory of continued fractions, one can prove that, for any 
irrational real number $\xi$, there are arbitrarily large integers $X$ such that
the inequalities
$$
\Vert q \xi || \le {1 \over 2X} \quad {\rm and} \quad 0 < q \le X    
$$
have no integer solutions; see \cite{Kh26b} or Proposition 2.2.4 of \cite{Bu16}. 
Consequently, there are no singular real irrational numbers and, a fortiori, no very well uniformly 
approximable real irrational numbers neither.   
However, there do exist real irrational numbers which are singular on average; see Section 4.

The proof of  \eqref{eqn:2.0} is based on a transference argument of \cite{BuLa05b} which relates 
classical Diophantine approximation properties of a matrix to its uniform 
inhomogeneous aproximation properties. 
In particular, the Theorem of \cite{BuLa05b} implies that
for every very well uniformly approximable matrix $A$, the set $\Bad^{\infty} (A)$ has full 
Lebesgue measure. 
Note also that for every singular matrix $A$, the set $\Bad^{\infty} (A)$ has full 
Hausdorff dimension. This was proven by Moshchevitin \cite{Mosh11} 
and, independently, by Einsiedler and Tseng \cite{EiTs11}. 
Note that it follows           
from Minkowski's theorem quoted in Section 1 that, for every 
real irrational number $\alpha$, the set $\Bad^\infty ((\alpha))$
is empty. 
 
We can partially extend Theorem \ref{thm:general} to the 
case $(n, m) \not= (1, 1)$. 
To describe our result, we first need to define the notion of best approximation vectors associated to 
a matrix $A = (\alpha_{i,j})$.    
We denote by
$$
M_j(\uy) = (^t A \uy )_j =    \sum_{i=1}^{n}\alpha_{i,j}y_i , \quad \uy ={}^t (y_1, \dots, y_n), 
\quad (1 \le j \le m)
$$
the linear forms  determined by the  columns of $A$ and we set
$$
M(\uy) = \Vert {}^tA\uy\Vert = \max_{1\le j\le m}\Vert M_j(\uy)\Vert .
$$
Observe that the quantity $M(\uy)$ is  positive for all non-zero integer
       $n$-tuples $\uy$,
since we have assumed that $\mathrm{rk}(G_A)=m+n$.
Thus, we can build inductively a sequence of integer vectors
$$
\uy_i = {}^t(y_{i,1}, \dots , y_{i,n}) , \quad (i\ge 1),
$$
called {\it a sequence of  best approximations}
related to the linear  forms
$M_1, \dots , M_m$ and to the supremum  norm,  which satisfies the
following properties:

\begin{enumerate}
\item Setting,
$
\vert \uy_{i }\vert = Y_i  \ \text{ and } \   M_i =  M(\uy_i),
$ we have
$$
1 = Y_1 < Y_2<\cdots  \ \text{ and } \    M_1 > M_2 > \cdots \enspace ,
$$

\item $M(\uy)\ge M_i$ for all non-zero integer vectors   $\uy$ of norm
$\vert \uy
\vert < Y_{i+1} $.
\end{enumerate}
We start the construction
with a  smallest {\it  minimal point } $\uy_1$ in the sense of \cite{DaSc69},
satisfying
$Y_1 = \vert \uy_{1}\vert = 1$
        and   $M(\uy)\ge M(\uy_1) =M_1$
for any  integer point  $\uy$ in $\bZ^n$ with norm $\vert \uy
\vert = 1$.
Suppose that   $\uy_1, \dots , \uy_i$ have already been constructed
in such a way that $M(\uy)\ge M_i$ for all non-zero integer point $\uy$
of norm
$\vert\uy\vert \le Y_i$. Let  $Y$
be the smallest positive integer  $>Y_i$ for which there exists an
integer point $\uz$
verifying  $\vert \uz \vert =Y$ and  $ M(\uz) < M_i$.
The integer  $Y$ does exist by Dirichlet box principle since
$M_i>0$.
Among those points $\uz$, we select an element $\uy$ for which
$M(\uz)$ is minimal.
      We then set
$$
\uy_{i+1} = \uy, \quad
Y_{i+1} = Y, \ \text{ and } \  M_{i+1} =  M(\uy).
$$
The sequence  $(\uy_i)_{i\ge 1}$ obtained in this way clearly satisfies the
desired properties.

Furthermore, as established along the proof of Lemma 1 of \cite{BuLa05b}, we have
$$
Y_{i + 3^{m+n}} \ge 2 Y_{i+1}, \quad i \ge 1.
$$
In the case $m=n=1$, the sequence of best approximations coincides with the 
sequence of denominators of convergents.

The implication (ii) $\Rightarrow$ (i)  of Theorem~\ref{thm:general} extends as follows. 

\begin{thm} \label{thm:1.6}
Let $A$ be an $n \times m$ matrix
and $(\uy_k)_{k \ge 1}$ a sequence of best approximation vectors associated to $A$. 
If $|\uy_k|^{1/k}$ tends to infinity with $k$, then there exists a positive real number $\eps$ such that 
$$
\dim_H \Bad^\eps (A) = n. 
$$
If, furthermore, $|\uy_{k+1}| / |\uy_k|$ tends to infinity, then $\eps$ can be taken
to be any positive real number less than $(4 n)^{-1} (4 m)^{-m/n}$.      
\end{thm}

We do not know whether the rest of Theorem~\ref{thm:general} extends to matrices, that is, 
whether the properties `$A$ is singular on average' and 
`$Y_k^{1/k}$ tends to infinity' coincide in dimension $m \times n$ 
with $(m, n) \not= (1, 1)$ and also if these conditions are equivalent to the existence of $\eps>0$ 
with $\dim_H \Bad^\eps(A)=n$.  

\section{Badly approximable numbers and the convergents}
In this section, we prove Theorem~\ref{thm:general} (i) $\Leftrightarrow$ (ii) and Theorem~\ref{thm:2s}. 
\subsection{Inhomogeneous approximation using homogeneous approximation} 
In this subsection, we use a result concerning the Hausdorff dimension 
in homogeneous Diophantine approximation to prove Theorem~\ref{thm:2s} 
and  implication (ii) $\Longrightarrow$ (i) of Theorem~\ref{thm:general} 

We start with a corollary of a theorem of Erd\H os and Taylor \cite{ErTa57}, of which we give a proof for the sake of completeness.

\begin{thm}\label{ET}
Fix $0< \delta <1/2$. 
Let $(n_k)_{k \ge 1}$ be an increasing sequence of integers such that 
$n_{k+1} / n_k \ge 4/(1-2\delta)$ 
for $k$ sufficiently large and $\lim n_k^{1/k}=\infty$. The set 
$$
{\cS}_{\delta} = \{x \in [0, 1] : \hbox{there exists $k_0 (x)$ 
such that } \| n_k x \| > \delta   \hbox{ for all $k \ge k_0 (x)$} \}.
$$
has Hausdorff dimension 1.
Moreover, if $\lim n_{k+1} / n_k=\infty$, then $\dim_H S_\delta =1$ for any $\delta$ in $(0, 1/2)$.
\end{thm}

\begin{proof}
Let $\delta$ be real with $0 < \delta < 1/2$. 
We consider the Cantor set ${\cS}_\delta := \cap_k E_{k, \delta}$, where 
$$
E_{k, \delta} = [0, 1] \cap \bigcup_{0 \le j \le n_k} \, 
\Bigl[ {j + \delta \over n_k}, {j + 1 - \delta \over n_k} \Bigr].
$$
The length of the intervals composing $E_{k, \delta}$ is equal to $(1-2\delta) n_k^{-1}$. 
The distance between two intervals in $E_{k, \delta}$ is $2\delta n_k^{-1}$. 
An interval composing $E_{k, \delta}$ contains at least $(1-2\delta) n_{k+1} / n_k - 2$ intervals 
composing $E_{k+1, \delta}$. For $k$ large enough, since 
$n_{k+1} \ge 4 n_k / (1-2\delta)$, we see that 
any interval composing $E_{k, \delta}$ contains at least $2$ intervals 
composing $E_{k+1, \delta}$. 
We are in position    
to apply the mass distribution principle.
By Example 4.6 of \cite{Fal90}, we obtain that 
\[
\dim_H {\cS}_\delta \ge \liminf_{k  \to\infty} \frac{ \log (m_1 m_{2} \cdots m_{k-1}) } {- \log ( m_k \ep_k )} ,
\]
where $m_k$ is the smallest number of intervals of $E_{k,\delta}$ in each interval of $E_{k-1,\delta}$,
and $\ep_k \ge  2\delta /n_k$ is the minimal distance between intervals of $E_{k,\delta}$. 
We check for sufficiently large $k$
$$m_k \ge \frac{ (1-2\delta) n_{k}}{n_{k-1}} -2 = \frac{ (1-2\delta) n_{k}}{2n_{k-1}}.$$
Thus we have
$$
\dim_H {\cS}_\delta \ge \liminf_{k\to\infty}  \, {\log n_{k-1} + k \log ( (1-2\delta)/2 )  \over \log n_{k-1}  
- \log (\delta (1 - 2\delta))}.
$$
Since $\lim n_k^{1/k}=\infty,$ 
under the assumption $n_{k+1} \ge 4 n_k / (1-2\delta)$ for $k$ large enough,
we have $\dim {\cS}_{\delta} = 1$.  

If $n_{k+1} / n_k$ tends to infinity, then for any given $\delta$ in $(0, 1/2)$, the 
assumption $n_{k+1} \ge 4 n_k /(1-2\delta)$ is satisfied for $k$ large enough. This shows that 
$\dim {\cS}_{\delta} = 1$.  
\end{proof}

{\it Proof of Theorem 1.2.}  
We apply the second assertion of Theorem~\ref{ET} to prove Theorem~\ref{thm:2s}. 
Let $\alpha$ be an irrational number and $(q_n)_{n \ge 1}$ the sequence of
denominators of its convergents.  Assume that $q_{n+1} / q_n$ 
tends to infinity (equivalently, that $a_n$ tends to infinity). 
Let $k$ be an integer. Let $x$ be in $(0, 1)$ and observe that, for every integer $y$, we have
\begin{equation}\label{eqn:2.1}
\|y x \| \le |y| \cdot \|\alpha k - x \| + |k| \cdot \|\alpha y\|.   
\end{equation}
Let $\delta$ with $0 < \delta < 1/2$. 
Assume that $|k|$ is large and let $\ell$ be the integer with
$$
q_{\ell} \le \frac{2}{\delta}  |k| < q_{\ell + 1}.
$$
Assume that $x$ is in ${\cS}_{\delta}$. 
Letting $y = q_{\ell}$ in \eqref{eqn:2.1}, we have

\begin{equation*}
|q_{\ell}| \cdot \|\alpha k - x \| \ge  \|q_\ell x \| - |k| \cdot \|\alpha q_{\ell}\|   
 \ge \delta - \frac{|k|}{q_{\ell+1}}  = \frac{\delta}{2}. 
\end{equation*}
This gives
$$
|k| \cdot \|\alpha k - x \| \ge {\delta^2 \over 4}.
$$
Since $\delta$ can be chosen arbitrarily close to $1/2$, the theorem is proved. 
\qed

\begin{thm}\label{thm:1.3}
Let $\alpha$ be a real number for which $q_k^{1/k}$ tends to infinity. 
We have $$
\dim_H \Bad^{1/ (2^4 \cdot 3^3)} (\alpha)= 1.
$$
\end{thm}

\begin{proof} 
We first claim that for each $R > 1$    
there exists an increasing function
$\varphi : \bZ_{\ge 1} \to \bZ_{\ge 1}$ satisfying $\varphi (1) \geq 1$
and, for any integer $i \ge 2$,
\begin{equation}\label{eqn:2.3}
q_{\varphi(i)} \ge R q_{\varphi (i-1)} \quad
{\rm and} \quad q_{\varphi(i-1) + 1} \ge  q_{\varphi (i)} / R.
\end{equation}
The function $\varphi$ is constructed in the following way.
Let 
$$\cJ_0=\{ j : 
q_{j+1} \ge R q_{j} \},
$$
which is an infinite set since $q_k^{1/k}$ tends to infinity. 
Let $\varphi(1)$ be the smallest element of $\cJ_0$.    
Suppose that we have defined $\varphi$ up to $\varphi(h)$ in $\cJ_0$.
We will define $\varphi(h+1), \cdots, \varphi(h')$ for some $h'$ which we will determine shortly. 
Define $\varphi(h')$ to be the smallest element of $\cJ_0$   
greater than $\varphi(h)$.    
 
Define
$\varphi(h'-1)$ to be the largest index $t > \varphi(h)$   
for which $q_{\varphi(h')} \ge R \, q_t$. We let
$\varphi(h'-2)$ be the largest index $t > \varphi(h)$    
for which $q_{\varphi(h'-1)} \ge R \, q_t$, and so on until it does not
exist any index $t$ as above.
Let us say that we have just defined  $\varphi(h'), \varphi(h'-1), \ldots,
\varphi(h'-h_0)$.
Define $h' = h_0 + h + 1$. It is easy to check that the
inequalities \eqref{eqn:2.3} are satisfied for $i= h+1, \ldots , h_0 + h+ 1$. The claim follows.

For a given $0< \delta <1/2$, let $R = \frac{4}{1-2\delta}$.
As $\varphi$ is increasing, $q_{\varphi(k)}^{1/k}$ tends to infinity. Thus, we can apply 
Theorem~\ref{ET} to the sequence $n_k = q_{\varphi(k)}$ 
to obtain a set ${\mathcal S}_{\delta}$. For $x$ in $S_\delta$, we argue as in the proof 
of Theorem~\ref{thm:2s}. 

Let $k$ be an integer.  Fix a positive real number $M$    
and let $\ell$ be the integer with 
$$
q_{\varphi(\ell)} \le M |k| < q_{\varphi(\ell + 1)}.
$$
Note that, by \eqref{eqn:2.3}, we have 
$$
\|\alpha q_{\varphi(\ell)}\|  \le q_{\varphi(\ell)+1}^{-1} 
\leq R q_{\varphi(\ell+1)}^{-1}.
$$
Thus by \eqref{eqn:2.1}, if $x$ is in $\cS_\delta$, then
\begin{align*}
|q_{\varphi(\ell)}| \cdot \|\alpha k - x \| 
& \ge  \|q_{\varphi(\ell)} x \| - |k| \cdot \|\alpha q_{\varphi(\ell)}\|   \\
& > \delta - R \cdot |k| \cdot q_{\varphi(\ell+1)}^{-1}  
 \ge \delta - \frac{R}{M}  
\end{align*}
This gives
$$
|k| \cdot \|\alpha k - x \| 
\ge \frac{1}{M} q_{\varphi(\ell)} ||\alpha k -x||
\ge \left( \delta - \frac{R}{M}\right) \frac{1}{M},
$$
which attains the minimum $(2^4 \cdot 3^3)^{-1}$ at $\delta=1/3$ and $M=72$. This completes the proof.
\end{proof}

\subsection{Non-singular on average}

In this subsection, we show the implication (i) $\Longrightarrow$ (ii) of Theorem~\ref{thm:general}. 
Let us assume that $C_0 := \liminf_k  \log q_k /k < \infty$ and show that 
$\dim_H \Bad^\ep (\alpha) <1.$
Since $\Bad^\ep(\alpha) =\underset{K \to \infty}{ \lim} \Bad^\ep_K (\alpha)$, where
$$\Bad^\ep_K (\alpha) :=\{x\in [0,1]\, : \, |q| \cdot \| q\alpha -x \| \geq \epsilon, \;  \forall q \geq q_K\}$$
is an increasing sequence of sets, it is enough to show that $\dim_H \Bad^\ep_K(\alpha)$ 
has a uniform upper bound smaller than 1.

Throughout the paper, for a positive integer $n$, we view $n \alpha$ as a point 
on the circle $S^1 = \mathbb R / \mathbb Z$ represented by an element in $[0,1)$. 
For convenience, given $a,b$ in $S^1,$ we denote by $(a,b)$ 
the shorter interval among $(a,b)$ and $(b,a)$ in $S^1$ regardless of whether $a<b$ or not. 
For any $k$ in $\mathbb{Z}_{\geq 0}$, define ${\cA}^{(k)}$ to be the set of connected components 
of $S^1 \setminus \{ \alpha, 2\alpha, \ldots, q_k \alpha \}$. We will call $\cA^{(k)}$ a partition of $S^1$ 
by ignoring the endpoints of the intervals. 
This partition consists of $q_k$ intervals, which we call $I_n^{(k)}$, of the following two types :
\begin{enumerate}
\item when $n< q_k-q_{k-1}$ : $I_n^{(k)}= (n\alpha, (n+q_{k-1}) \alpha)$ and its length is $\| q_{k-1} \alpha \|$.
\item when $n> q_k-q_{k-1}$ : $I_n^{(k)}= (n\alpha, (n+q_{k-1}-q_k) \alpha)$ 
and its length is $\| q_{k-1} \alpha \|+ \| q_k \alpha \|$.
\end{enumerate}
Their lengths are bounded as follows: 
\begin{equation}\label{eqn:bound}
\frac{1}{2q_k} < \| q_{k-1} \alpha \| < \frac 1{q_k} \quad \hbox{and} \quad 
\frac 1{q_k} < \| q_{k-1} \alpha \| + \| q_k \alpha \| < \frac {2}{q_k}.
\end{equation} 
Let us denote by $|n|_q$ the integer in $[1, q]$ which is congruent to $n$ modulo $q$.  
With this notation, the elements of ${\cA}^{(k)}$ are 
the intervals $I_n^{(k)} := [ n \alpha,  |n+q_{k-1}|_{q_k} \alpha ]$ with $n=1,2, \ldots, q_k$.

The partition ${\cA}^{(k+1)}$ is clearly a refinement of the partition ${\cA}^{(k)}$. 
Every element of ${\cA}^{(k)}$ is divided into either $a_{k+1}$ or $a_{k+1}+1$ elements of ${\cA}^{(k+1)}$. 
In particular, let $m$ be an integer with $q_k < m \leq q_{k+1}$,
then $m\alpha$ is in $I_n^{(k)}$ if and only if 
\begin{equation}\label{eqn:2.4}
m = n + q_{k-1} + cq_k.
\end{equation}
for some integer $0 \le c \le a_{k+1} -1.$ See Figure~\ref{fig1}. 

\begin{figure}
\begin{center}
\begin{tikzpicture}[every loop/.style={}]
  \tikzstyle{every node}=[inner sep=-1pt]
  \node (-1) at (-1,0) {$[$} ;
  \node (0) at (0,0) {$|$} node [below=15pt] at (-1,0)  {$(n+q_{k-1})\alpha$};
  \node (1) at (1,0) {$|$} node [above=15pt] at (1,0) {$(n+q_k+q_{k-1})\alpha$};
  \node (2) at (2,0) {$|$};
  \node (3) at (3,0) {$|$};
  \node (5) at (5,0) {$|$} node [below=15pt] at (5,0) {$m \alpha $};
  \node (6) at (5.5,0) {$\bullet$} node [below=6pt] at (5.5,0) {$x$};  
  \node (7) at (6,0) {$|$} node [above=15pt] at (6,0) {$ (m + q_k) \alpha$};
  \node (10) at (10,0) {};
  \node (11) at (11,0) {$]$} node [below=15pt] at (11) {$n\alpha$};
  \path[-] (-1) edge (11);
  \draw [line width=.5mm,   black] (5) -- (7);
\end{tikzpicture}
\end{center}
\caption{When $k$ is even, $x$ is in $I^{(k)}_{n} \cap I^{(k+1)}_{m}$, $m = b q_k + q_{k-1} +n$.} 
\label{fig1}
\end{figure}
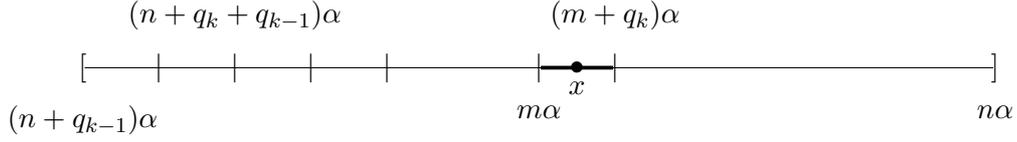

We define a sequence $(k_i)_{i \ge 0}$ 
in the following way: set $k_0=K$ and, for $i \ge 0$, let $k_{i+1}$ be the smallest integer for which 
$q_{k_i}/q_{k_{i+1}} < \ep/12$.   
Observe that since $q_{k+2}>2q_k$, the sequence $(k_{i+1}-k_i)_{i \ge 0}$ 
is uniformly bounded from above by a contant which we denote by $C_\ep$. Hence we obtain 
\begin{equation} \label{eqn:liminf}
\liminf_{i\to\infty} \frac 1i \log q_{k_i} < C_0 C_\ep . 
\end{equation}

Note that
\begin{align*}
\Bad^\ep_{K}(\alpha) &= \bigcap_{i=0}^\infty \underset{q_{k_i} \leq n < q_{k_{i+1}}}{\bigcap}  
B\Bigl(n\alpha, \frac{\ep}{n} \Bigr)^c  
 \subset   \bigcap_{i=0}^\infty \underset{q_{k_i} \leq n < \frac{\ep}{2} q_{k_{i+1}}}{\bigcap} 
 B \Bigl(n\alpha, \frac{\ep}{n} \Bigr)^c  \\
& \subset   \bigcap_{i=0}^\infty \underset{q_{k_i} \leq n < \frac{\ep}{2} q_{k_{i+1}}}{\bigcap} \left( I_n^{k_{i+1}} \right)^c. 
\end{align*}
The last inclusion above follows from the fact that each interval $I_n^{(k_{i+1})}$ 
is contained in $B(n\alpha, \frac{\ep}{n})$ since it has one endpoint $n\alpha$ 
and is of length at most $2/q_{k_{i+1}} < \ep/n.$

Thus, by letting $F_K=\bigcap_{i \ge 0} F_K^{(i)}$, where 
$$
\cS_j = \underset{q_{k_j} \leq n      
< \frac{\ep}{2} q_{k_{j+1}}}{\bigcap} \left( I_n^{(k_{j+1})} \right)^c     
\qquad \textrm{and} \qquad 
F_K^{(i)} 
= \underset{0 \leq j \leq i}{\bigcap}  
\cS_j,     
$$
it is enough to find a uniform upper bound for $\dim_H  F_K.$

We need the following lemma to estimate the number of subintervals 
of $I_{m}^{(k_{i+1})}$ in each $I_{n}^{(k_i)}$.

\begin{lem}\label{lem_subintervals}
Let $Q \ge 6q_k$. 
For each interval $I_n^{(k)}$ of $\cA^{(k)}$, the number of points $m\alpha$ 
which belong to $I_n^{(k)}$ for $q_k < m \le Q$ is 
at least equal to $Q / 4q_k$. 
\end{lem}

\begin{proof} Let $c$ be the positive integer defined by the inequalities 
$$
q_k + c (2q_k) \le Q  < q _k + (c+1) (2q_k).
$$
Each interval $I_n^{(k)}$ contains at least one point $m\alpha$ with  
$$ 
(2i-1) q_k + 1 \le m \le (2i+ 1)q_k, \qquad  1 \le i \le c, 
$$
since the length of interval $I_n^{(k)}$ is at least $\| q_{k-1} \alpha\|$ 
and the distance between   
two points of any neighboring point of 
$m\alpha$ with $ (2i-1) q_k + 1 \le m \le (2i+ 1)q_k$ is at most $\| q_{k-1} \alpha\|$.

Hence, the number of points $m\alpha$ contained in $I_n^{(k)}$ with $q_k + 1 \le m \le (2c+1)q_k \le Q$ 
is at least $c$. 
Using the assumption $Q \ge 6q_k$, we have
$$c > \frac{Q - 3q_k}{2q_k} \ge \frac{Q}{4q_k}. $$
\end{proof}
By Lemma~\ref{lem_subintervals}, for 
each interval $I_{n}^{(k_i)}$ in $\cS_i$, the number of intervals
$I_m^{(k_{i+1})}$ in $I_{n}^{(k_i)}$ 
which contain a point $m\alpha$ with $q_{k_i} \le m \leq \ep q_{k_{i+1}} /2 $ 
is at least $\ep q_{k_{i+1}}/ (16 q_{k_i})$, since at most two points belong to one interval of $F_K^{(i+1)}$.

Since the total number of intervals $I_m^{(k_{i+1})}$ 
contained in an interval of $I_{n}^{(k_i)}$ in $\cA^{(k_i)}$   
is at most 
$$
\frac{\| q_{k_i-1} \alpha\| +\| q_{k_i} \alpha\|}{\| q_{k_{i+1}-1} \alpha\| } < \frac{4 q_{k_{i+1}}}{q_{k_i}},
$$
for each interval $I_{n}^{(k_i)}$ of $\cA^{(k_i)}$,  we have 
$$ \frac{\# \left\{  I_m^{(k_{i+1})} \notin \cS_{i+1}  :  I_m^{(k_{i+1})}  \subset  I_{n}^{(k_i)}  \right \} }{\# \left\{  I_m^{(k_{i+1})} \in \cA^{(k_{i+1})}  :  I_m^{(k_{i+1})}  \subset  I_{n}^{(k_i)}  \right \} }  > \frac{ \ep q_{k_{i+1}}/ (16 q_{k_i}) } {4 q_{k_{i+1}}/q_{k_i}}$$    
thus  
$$   
\frac{\# \left\{  I_m^{(k_{i+1})} \in \cS_{i+1}  :  I_m^{(k_{i+1})}  \subset  I_{n}^{(k_i)}  \right \} }{\# \left\{  I_m^{(k_{i+1})} \in \cA^{(k_{i+1})}  :  I_m^{(k_{i+1})}  \subset  I_{n}^{(k_i)}  \right \} }
< 1- \frac{ \ep q_{k_{i+1}}/ (16 q_{k_i}) } {4 q_{k_{i+1}}/q_{k_i}} =1- \frac{\ep}{32}.
$$
Since the cardinality of $\cA^{(k_i)}$ is $q_{k_i}$, the number of intervals $I_{n}^{(k_i)}$ in $\cS_i$ is at most
$$
q_{k_i} \left( 1 - \frac{\ep}{32} \right)^i.
$$
Thus, for any $s \leq 1$, we have
\[
\sum_{I_n^{(k_i)}\in \cS_i}  \left| I_n^{(k_i)} \right|^s  < q_{k_i}   
\left( 1 - \frac{\ep}{32} \right)^i \cdot \left(\frac {2} {q_{k_i}}\right)^s.
\]
By \eqref{eqn:liminf}, for any $M> C_0 C_\ep$, there exists a sequence $k_i$ tending to infinity  
for which $q_{k_i}\leq M^i$.  
Since 
$
\underset{k_i : q_{k_i}\leq M^i}{\bigcap} F_K^{(i)}
$
is a covering of $F_K$, we obtain 
\[
\dim_H F_K \leq 1 + \frac {\log (1-\ep/32)}{\log  M}.
\]
This completes the proof of the lemma.

\section{Approximation of one-sided limit}

In this section, we assume that the sequence $(a_k)_{k \ge 1}$ of partial quotients 
of $\alpha$ tends to infinity and prove Theorem \ref{thm:14}.
Recall that the partition ${\cA}^{(k)}$   
of $S^1$ consists of the intervals 
 $I_n^{(k)} := [ n \alpha,  |n+q_{k-1}|_{q_k} \alpha ]$ with $n=1,2, \ldots, q_k$.

Recall also that the numbers $1 \leq m \leq q_{k+1}$ for which the corresponding 
point $m \alpha$ is contained in $I_n^{(k)}$ are (in order, looking from 
$|n+q_{k-1}|_{q_k} \alpha$ towards $n \alpha$) $n+q_{k-1}+q_k$, $n+q_{k-1}+2q_k$, etc.

If $I_{n}^{(k)} \supset I_{m}^{(k+1)}$ with $1\le n \le q_{k}$, $1\le m \le q_{k+1}$, then we have
\begin{equation}\label{eqn:range}. 
m = n+   q_{k-1} + b q_k , 
\end{equation}
for some integer $b$ with $-1 \le b \le a_{k+1}-1 $. 
Thus two endpoints $m\alpha$, $n\alpha$ of $I_{m}^{(k+1)}, I_{n}^{(k)}$ are separated by  
\begin{equation}\label{e1}
\| m\alpha - n\alpha \| = \| (bq_k + q_{k-1}) \alpha \| = \| q_{k-1} \alpha \| - b\| q_k \alpha \|.
\end{equation}

To prove the first part of Theorem \ref{thm:14}, let us fix $\ep<1/4$. 
Choose $K$ large enough so that $a_k$ is large for every $k \ge K$  
and select some sequence $(\gamma_k)_{k \ge 1}$ which tends to $0$ as $k$ tends to infinity.  
We will later specify the conditions satisfied by $\gamma_k$.

Throughout this section, set $\delta_k := \frac{q_{k-1}}{q_{k}}$. Since 

$$ q_k \| q_{k} \alpha \| + q_k \| q_{k-1} \alpha \| = 1 + (q_k - q_{k-1} ) \| q_k \alpha \|  < 1 + \frac{q_k q_{k+1}\| 
q_k \alpha \| } {q_{k+1}} < 1 + \delta _{k+1}$$
and
$$ q_k \| q_{k-1} \alpha \| = 1 - q_{k-1} \| q_k \alpha \| > 1 - \frac{q_{k-1}}{q_k} \frac{q_k}{q_{k+1}}  q_{k+1}\| q_k 
\alpha \|  > 1 - \delta_{k} \delta_{k+1},$$
we get 
$$
\| q_{k} \alpha \| + \| q_{k-1} \alpha \|  < \frac{1 + \delta_{k+1}}{q_k}, \qquad \| q_{k-1} \alpha \|  
> \frac{1-\delta_{k} \delta_{k+1}}{q_k}.
$$

\begin{lem} \label{lem:induc}
Let $K$ be a natural number such that, for $k \ge K$, we have 
\[
\gamma_k + 2\delta_k < 1-2\ep^{1/2}.
\]
Let $k \ge K$. Then for every $n_1, n_2$ with
\begin{equation}\label{eqn:3.3}
\ep^{1/2}q_k + q_{k-1} < n_1 < (\ep^{1/2}+\gamma_k)q_k + q_{k-1}
\end{equation}
and 
\begin{equation}\label{eqn:3.4}
\ep^{1/2}q_{k+1} + q_k < n_2 < (\ep^{1/2}+\gamma_{k+1})q_{k+1} + q_k,
\end{equation}
if $I_{n_2}^{(k+1)} \subset I_{n_1}^{(k)}$, then $I_{n_2}^{(k+1)}$ 
is disjoint from    
all the balls $B(n\alpha, \ep/n)$ such that
\begin{itemize}
\item[(a)] $n \alpha$ is an endpoint of $I_{n_1}^{(k)}$, or
\item[(b)] $q_k< n \leq q_{k+1}$ and $n \alpha$ is not an endpoint of $I_{n_2}^{(k+1)}$.
\end{itemize}
\end{lem}

\begin{proof}
As $(n_1 + q_{k-1})\alpha$ and $(n_2 + q_{k})\alpha$ 
are endpoints of $I_{n_1}^{(k)}$ and $I_{n_2}^{(k+1)}$, for part (a), we need to check the inequalities
$$
\frac{\ep}{n_1} <  \| (n_2 + q_k)\alpha - n_1 \alpha  \| 
\;\;\; \mathrm{and}\;\;\;\;
\frac{\ep}{n_1 + q_{k-1}} <  \|n_2 \alpha  - (n_1+q_{k-1}) \alpha \|.
$$
Letting $n_2 = bq_k + q_{k-1} + n_1$ as in \eqref{eqn:range}, we have  
$$
\| n_2 \alpha - n_1 \alpha \| = \| (bq_k + q_{k-1}) \alpha \| = \| q_{k-1} \alpha \| - b \| q_k \alpha\|.
$$
Therefore, we have
\begin{equation}\label{3.3}
\begin{split}
\| (n_2 + q_k)\alpha - n_1 \alpha \| &=  \| q_{k-1} \alpha \| - (b+1) \| q_k \alpha\|  \\
&= \frac{1 - q_{k-1}\| q_{k} \alpha \|}{q_k} - \left(\frac{n_2 - q_{k-1} - n_1 + q_k}{q_k} \right) \| q_k \alpha\| \\
&> \frac{1}{q_k} - \left(\frac{n_2 + \delta_{k+1}q_{k+1}}{q_k}\right) \| q_k \alpha\| 
> \frac{1 - \ep^{1/2} - \gamma_{k+1} - 2\delta_{k+1}}{q_k}
> \frac{\ep^{1/2}}{q_k},
\end{split}
\end{equation}
where the last inequality follows from \eqref{eqn:3.3}. We also have
\begin{equation}\label{3.4}
\begin{split}
|n_2 \alpha  - (n_1+q_{k-1}) \alpha | &= b \| q_k \alpha\| = \frac{n_2 - q_{k-1} - n_1}{q_k} \| q_k \alpha\| \\
&> \left( \frac{\ep^{1/2}q_{k+1} + (1 - \ep^{1/2} -\gamma_{k}) q_{k} - 2q_{k-1} }{q_k} \right) 
\left( \frac{1 - \delta_{k+1} \delta_{k+2}}{q_{k+1}} \right) \\
&> \frac{\ep^{1/2} \left ( 1 + \delta_{k+1} \right) (1 - \delta_{k+1} \delta_{k+2}) }{q_k}  >  \frac{\ep^{1/2}}{q_k}.
\end{split}
\end{equation}
Therefore, we have 
$$
\frac{\ep}{n_1} <  \frac{\ep}{\ep^{1/2}q_{k} + q_{k-1}} < 
\frac{\ep^{1/2}}{q_k} < | (n_2 + q_k)\alpha - n_1 \alpha |$$
and
$$
\frac{\ep}{n_1 + q_{k-1}} < \frac{\ep}{\ep^{1/2}q_{k} + 2q_{k-1}}
 < \frac{\ep^{1/2}}{q_k} < |n_2 \alpha  - (n_1+q_{k-1}) \alpha |.
$$

\medskip

For part (b), we separate two cases: 

(i) If $n\alpha$ is in $I_{n_1}^{(k)}$ and $n \alpha$ is not an endpoint of $I_{n_2}^{(k+1)}$, 
then  $n = n_2 - dq_k$ or $n = n_2 + (d+1) q_k$ for some $d \ge 1$.

Suppose that $n = n_2 - dq_k$. Then, by the condition $n > q_k$, we have 
$$1 \le d < \frac{n_2}{q_k} -1.$$  
Thus, combined with \eqref{eqn:3.4}, we deduce that 
$$ d \left( \frac{n_2}{q_k} - d \right) \ge \left( \frac{n_2}{q_k} - 1 \right)
> \frac{\ep^{1/2}q_{k+1}}{q_k} > \frac{2 \ep q_{k+1}}{q_k} > \frac{\ep}{q_k \| q_k \alpha \|}.
$$
Hence, for $n = n_2 - d q_k$, we have 
$$
\frac{\eps}{n} = \frac{\ep}{n_2 - d q_k} <  d \| q_k \alpha \| = \left| n\alpha - n_2\alpha \right|.
$$
When $n = n_2 + dq_k$, we get 
$$
\frac{\eps}{n} = \frac{\ep}{n_2 + (d +1)q_k} <  d \| q_k \alpha \| = \left| n\alpha - (n_2 + q_k) \alpha \right|.
$$

(ii) If $n\alpha$ is not in $I_{n_1}^{(k)}$, 
then the distance between $n\alpha$ and $I_{n_2}^{(k)}$ is bigger than 
$ | (n_2 + q_k)\alpha - n_1 \alpha |$ and $|n_2 \alpha  - (n_1+q_{k-1}) \alpha | $.
By \eqref{3.3} and \eqref{3.4}, we have 
$$ \min\{  |n\alpha - n_2 \alpha |,  |n\alpha - (n_2 +q_k) \alpha | \} >  \frac{\ep^{1/2} }{q_k} 
> \frac{\ep}{q_k} > \frac{\ep}{n}.$$
\end{proof}

Denote by $F$ the set of all the points $x$ in $S^1$ such that, 
for all $k\geq K$, we have $x$ in $I_n^{(k)}$ 
with $\ep^{1/2}q_k + q_{k-1} < n <(\ep^{1/2}+ \gamma_k)q_k+ q_{k-1}$. 
By Lemma~\ref{lem:induc}, 

\begin{equation} \label{eqn:geq}
\liminf_{n\to \infty} n \| n\alpha -x \| \geq \ep.
\end{equation}

\begin{lem} \label{lem:check}
Under the assumptions of Lemma \ref{lem:induc}, 
\[
I_{ n_2 }^{(k+1)} \subset B \left( ( n_1 +q_{k-1}) \alpha, \, \frac {\ep_k} {n_1+q_{k-1}} \right),
\]
where
$
\ep_k := \left( \ep^{1/2} + \gamma_{k} + 2\delta_{k}\right) \left( \ep^{1/2} + \gamma_{k+1} + 2\delta_{k+1}\right).
$
\end{lem}

\begin{proof}
Since $n_2 = bq_k + q_{k-1} + n_1$, we have
\begin{equation*}
\begin{split}
| (n_2 + q_k)\alpha - (n_1+q_{k-1}) \alpha | &= (b+1) \| q_k \alpha\| = \frac{n_2 +q_k - q_{k-1} - n_1}{q_k} \| q_k \alpha\| \\
&< \left( \frac{ (\ep^{1/2} + \gamma_{k+1}) q_{k+1} + 2q_k }{q_k} \right) \left( \frac{1}{q_{k+1}} \right) \\
&= \frac{ \ep^{1/2} + \gamma_{k+1} + 2\delta_{k+1} }{q_k} \\ 
& =  \frac{\ep_k}{ (\ep^{1/2} + \gamma_k) q_{k} + 2q_{k-1}}
< \frac{\ep_k}{n_1 + q_{k-1}}. 
\end{split}
\end{equation*}
\end{proof}

Since $\gamma_k, \delta_k$ both tend to $0$ as $k$ tends to infinity,  
the sequence $(\ep_k)_{k \ge 1}$ tends to $\ep$. Thus, Lemma \ref{lem:check} implies that

\begin{equation} \label{eqn:leq}
\liminf_{n \to \infty} n \| n\alpha -x \|\leq \ep.
\end{equation}

By \eqref{eqn:leq} and \eqref{eqn:geq}, the set $F$ is contained
in $\{ x : \liminf_n n \| n\alpha - x \| = \varepsilon  \}$.

\begin{lem} \label{lem:dimh}
If
\[
\lim_{k\to\infty} \frac {\log \gamma_k} {\log a_k} =0
\]
then $\dim_H F=1$.
\end{lem}

\begin{proof}
Let 
$$F_k = \bigcup_{\ep^{1/2}q_k + q_{k-1} < n <(\ep^{1/2}+ \gamma_k)q_k+ q_{k-1}}  I_n^{(k)}. $$ 
Then $F = \cap_{k \ge K} F_k$.
We may assume that $K$ is large enough to ensure   
that $\gamma_k > \delta_k$ for all $k \ge K$.

Each  $F_k$ is a union of $q_{k-1}$ intervals of length at least   
$\left\lfloor \frac{ \gamma_k q_k}{q_{k-1}}  \right \rfloor \| q_{k-1} \alpha \|$ which are separated by at least 
$$ \| q_{k-2} \| -\left\lceil \frac{ \gamma_k q_k}{q_{k-1}}  \right \rceil \| q_{k-1} \alpha \|.$$
By Example 4.6. of \cite{Fal90}, we obtain that 
\[
\dim_H F \ge \liminf_{k  \to\infty} \frac{ \log (m_K m_{K+1} \cdots m_{k-1}) } {- \log ( m_k \ep_k )} ,
\]
where $m_k$ is the smallest number of intervals of $F_{k}$ in each interval of $F_{k-1}$
and $\ep_k$ is the minimal the distance between intervals of $F_k$. 
Then we have 
$$
m_k > \left( \frac{\gamma_{k-1} q_{k-1}}{q_{k-2}} -1 \right) 
= \frac{q_{k-1}}{q_{k-2}}(\gamma_{k-1}-\delta_{k-1}),  \qquad m_K = q_{K-1},
$$
and 
$$\
\ep_k \ge  \| q_{k-2} \| - \left (  \frac{ \gamma_k q_k}{q_{k-1}}  +1 \right)  \| q_{k-1} \alpha \| 
\ge (1-\gamma_k - \delta_k) \| q_{k-2}\alpha \| > \frac{1-\gamma_k-\delta_k}{2q_{k-1}}.
$$
Since $\log \gamma_k / \log a_k$ goes to zero, 
 we have
\[ 
\dim_H F \ge 
 \liminf_{k  \to\infty} \frac {\log q_{k-2} + \sum_{i=K}^{k-2} 
 \log ( \gamma_i - \delta_i)} {\log q_{k-2} - \log \gamma_{k-1} - \log(1-\gamma_k -\delta_k)+\log 2} = 1.
\]
\end{proof}

Thus, the preceding three lemmas prove the first part of Theorem \ref{thm:14}.

\medskip

Now  we prove the second part of Theorem \ref{thm:14}. 

\begin{lem} \label{lem:estim}
Let $\ep>1/4$. Then the set 
$\Bad^\ep_+(\alpha)$ is empty.  
\end{lem}

\begin{proof}
Suppose that $\Bad^\ep_+(\alpha)$ is nonempty and let $x$ be in this set. 
For any small positive $\delta$ 
with $\ep- 2\delta > 1/4$, we can choose $K$ large enough that 
$n \| n\alpha -x \|\geq \ep - \delta$ for any $n>q_K$
and $\delta_{k+1} + \delta_{k+1}^2 < \delta$ for $k \ge K$.

For $k=K, K+1,\ldots,$  denote by $I_{n_k}^{(k)}$ the element
of the partition ${\cA}^{(k)}$ containing $x$. 

For $n_{k+1} = n_k + c q_k + q_{k-1}$, 
the conditions  $x \notin B(n_k\alpha, (\ep -\delta)/n_k)$ and $x  \in I_{n_{k+1}}^{(k+1)}$ 
imply that 
\begin{align*}
\frac{\ep -\delta}{n_k} 
&< \left| x - n_k \alpha \right|  \le \left|  n_{k+1} \alpha - n_k \alpha \right| = |(cq_k + q_{k-1})\alpha|  \\
&=  \| q_{k-1}\alpha \| - c \| q_k \alpha \| = (a_{k+1} -c) \| q_k\alpha\| + \| q_{k+1}\alpha \|
<  \frac{a_{k+1} -c}{q_{k+1}} + \frac {1}{q_{k+2} }.
\end{align*}
Therefore,
$$n_{k+1} = n_k + cq_k + q_{k-1} < n_k + a_{k+1} q_k + q_{k-1} 
+ \frac{  q_k q_{k+1}}{q_{k+2}} - \frac{(\ep -\delta) q_k q_{k+1}}{n_k},$$
which implies  
$$ \frac{n_{k+1}}{q_{k+1}} < 1 + \frac{n_k}{q_{k+1}} + \frac{q_k}{q_{k+2}} - \frac{(\ep -\delta) q_k}{n_k} 
< 1 + 2\delta_{k+1} - (\ep -\delta) \frac{q_k}{n_k}.$$
Therefore,  for $k \ge K$, we get 
\begin{align*}
 \frac{n_k}{q_k} - \frac{n_{k+1}}{q_{k+1}}  &> - 1 - 2\delta_{k+1} + (\ep -\delta) \frac{q_k}{n_k} + \frac{n_k}{q_k} \\
&=  \frac{q_k}{n_k}  \left(  \left( \frac{n_k}{q_k}\right)^2 - (1 + 2\delta_{k+1})\frac{n_k}{q_k}  + (\ep -\delta)  \right) \\
&=  \frac{q_k}{n_k}  \left(  \left( \frac{n_k}{q_k} - \frac12 - \delta_{k+1} \right)^2  
+ \ep -\delta - \frac 14 - \delta_{k+1} - \delta_{k+1}^2  \right) \\
&\ge \ep -2\delta - \frac 14 > 0.  
\end{align*}
Thus,  we have 
$$ \frac{n_k}{q_k} < \frac{n_K}{q_K} - (k-K) \left( \ep -2\delta - \frac 14 \right) \to -\infty \ \text{ as } \ k \to \infty, $$   
which contradicts that $0 < n_k / q_k \le 1$.   
\end{proof}

\section{On real irrational numbers which are singular on average}

In this section we consider the case $n = m = 1$ and 
characterize the $1 \times 1$ matrices $(\alpha)$ 
which are singular on average, which gives the proof of the equivalence between (ii) and (iii) of Theorem~\ref{thm:general}.

\begin{prop}
Let $\alpha$ be a real number and $(p_k/q_k)_{k \ge 1}$ the sequence of its convergents. 
Then, $\alpha$ is singular on average if and only if $(q_k)^{1/k}$ tends to infinity 
with $k$. 
\end{prop}

\begin{proof}
Let $0< c < 1/2$ and let $k \ge 3$ be an integer. 
By the classical theory of continued fractions, we have     
$$
\min_{0 < n < q_{k+1}} \| n \alpha \| = \| q_k \alpha\|.
$$ 
Therefore, for each integer $X$ with $q_k \le X < q_{k+1}$,
the inequalities
\begin{equation}\label{eqn:1.1bis}
\Vert x \alpha \Vert \le c \, X^{-1}        
\quad {\rm and} \quad 0 < \vert x \vert \le X     
\end{equation}
have a solution 
if and only if $\Vert q_k \alpha \Vert \le  c X^{-1}$.  
Thus, for each integer $\ell$ in $[\log_2 q_k, \log_2 q_{k+1})$ 
the inequalities \eqref{eqn:1.1bis} have no solutions for $X = 2^\ell$
if and only if 
$$ 
- \log_2 \left( \| q_{k} \alpha \|/c \right)  < \ell < \log_2 q_{k+1}.
$$

Since $\|q_k \alpha \| < 1/q_{k+1}$,
the number of integers $\ell$ in $[\log_2 q_k, \log_2 q_{k+1})$ 
such that \eqref{eqn:1.1bis} have no solutions for $X = 2^\ell$ is at most 
$$
\left \lceil \log_2 q_{k+1}  +  \log_2 \left( \| q_{k} \alpha \| / c \right) \right \rceil
< \log_2 q_{k+1}  +  \log_2 \left( \| q_{k} \alpha \| / c \right) + 1 <  \log_2 (1/c) + 1.
$$
Hence, for an integer $N$ with 
$\log_2 q_k \le N < \log_2 q_{k+1}$, the number of integers $\ell$ in $\{1, \ldots, N\}$ 
such that \eqref{eqn:1.1bis}  
have no solutions for $X = 2^\ell$ is bounded from above by $(\log_2 (1/c) + 1)(k+1)$,
 thus 
\begin{multline*}
 \frac 1N  \Card \{ \ell \in \{1, \ldots, N\} : \mathrm{\; inequalities \;} \eqref{eqn:1.1bis} 
\mathrm{\;has \;no \;solution \;for} \; X=2^\ell \}  \\
\le \frac {(\log_2 (1/c) + 1)(k+1)}{N}  \le \frac {(\log_2 (1/c) + 1)(k+1)}{\log q_k}, 
\end{multline*}
which converges to 0 as $k$ goes to infinity, as soon as $(q_h)^{1/h}$ tends to infinity. 
Therefore, $\alpha$ is singular on average if $(q_h)^{1/h}$ tends to infinity.

Suppose that  $\alpha$ is singular on average.  
Choose $c = 1/4$.
Let $\ell$ be an integer satisfying $\log_2 q_{k+1} -1 \le \ell < \log_2 q_{k+1}$ for some $k \ge 1$.
Then, we have
$$ 
\| q_k \alpha \| > \frac{1}{2q_{k+1}} = \frac{2c}{q_{k+1}} > \frac{c}{2^\ell}.
$$
Since $\|n \alpha \| \ge \| q_k \alpha \|$ for any $0 < n < q_{k+1}$, 
we conclude that 
\eqref{eqn:1.1bis} have no solutions for $X = 2^\ell$ 
if $\ell$ is an integer in $[\log_2 q_{k+1} -1 , \log_2 q_{k+1})$. 
Recall that that $q_{k+1} \ge 2q_{k-1}$, thus the intervals 
$[\log_2 q_{k-1} -1 , \log_2 q_{k-1})$ and $[\log_2 q_{k+1} -1 , \log_2 q_{k+1})$ are disjoint.
Let $N$ be an integer with $\log_2 q_{2k} \le N < \log_2 q_{2(k+1)}$.
Since the intervals 
$$
[\log_2 q_2 -1, \log_2 q_2), [\log_2 q_4 -1, \log_2 q_4), \dots , [\log_2 q_{2k} -1, \log_2 q_{2k})
$$
are disjoint,
the number of integers $\ell$ in $\{1, \ldots, N\}$ 
such that \eqref{eqn:1.1bis} have no solutions for $X = 2^\ell$ and $c=1/4$ is at least $k$.
Hence, we have
\begin{align*}
\frac {k}{ \log q_{2k+2}} \le \frac {k}{N} 
< \frac 1N  \Card \{ \ell \in \{1, \cdots, N\} :  \eqref{eqn:1.1bis} \mathrm{\;have \;no \;solutions \;for} \; X=2^\ell, c=1/4  \} 
\end{align*}
and the condition of singularity on average implies that the right hand side 
of the inequality goes to 0 as $N$ goes to infinity.
By the monotonicity of $(q_k)_{k \ge 1}$, we deduce that $(q_k)^{1/k}$ goes to infinity.  
\end{proof}

\section{Proof of Theorem~\ref{thm:1.6}}

The key ingredient for the proof of Theorem~\ref{thm:1.6} is the following statement.

\begin{thm}  \label{thm:5.1}
Let $n \ge 1$ be an integer. 
Let $(\bfy_k)_{k \ge 1}$ be a sequence of integer vectors such that    
$|\bfy_{k+1}| / |\bfy_k| \ge 4n /(1-2\delta) +1$ 
for $k \ge 1$, where $| \cdot |$ is the $L_2$-norm on $\R^n$.
Assume that $|\bfy_k|^{1/k}$ tends 
to infinity with $k$ (or some other suitable condition). 
Then, for any $\delta$ in $(0, 1/2)$, setting
$$
{\cS}_{\delta} = \{\bfx \in [0, 1]^n : \hbox{there exists $k_0 (\bfx)$ such that } 
\| y_{k,1} x_1 + \ldots + y_{k,n} x_n \| > \delta        
\hbox{ for all $k \ge k_0 (\bfx)$} \}, 
$$
we have
$$
\dim_H  {\cS}_{\delta} = n. 
$$
\end{thm}

\begin{proof}
For $k \ge 1$ and $\delta$ in $(0, 1/2)$, set
$$ 
E_{k, \delta} = \{\bfx = (x_1, \ldots , x_n) \in [0, 1]^n : \| y_{k,1} x_1 + \ldots + y_{k,n} x_n \| > \delta  \}.
$$
Let $h$ be an index such that 
$$
| y_{k,h}| = \max_{ 1 \le i \le n} | y_{k,i}|.  
$$ 
For every $(n-1)$-tuple $( j_1, \dots, j_{h-1},  j_{h+1}, \dots, j_n$) of integers 
from $\{0, 1, \ldots , | y_{k,h}| - 1\}$,
there exist an integer $p$ and a real number $t$ with $0 \le t < 1/ | y_{k,h}|$,   
depending on $j_1, \dots, j_{h-1},  j_{h+1}, \dots, j_n$, such that  
$$
y_{k,1} \frac{j_1}{| y_{k,h}|} + \dots + y_{k,h-1} \frac{j_{h-1}}{| y_{k,h}|} 
+ y_{k,h} t + y_{k,h+1} \frac{j_{h+1}}{| y_{k,h}|} + \dots + y_{k,n} 
\frac{j_n}{ y_{k,h}|}  = \frac 12 + p. 
$$

For each integer vector $\mathbf j =  ( j_1, \dots , j_n)$ with $0 \le j_1, \dots , j_n < |\mathbf y_k^{(h)}|$, we write
$$
\mathbf w_k(\mathbf j) = \left(  \frac{j_1}{| y_{k,h}|}, \dots, \frac{j_{h-1}}{| y_{k,h} |}, 
\frac{j_{h}}{|y_{k,h}|} + t ,\frac{j_{h+1}}{|y_{k,h}|}, \dots , \frac{j_n}{| y_{k,h} |} \right) .
$$
Then, there exists $\eta$ in $\{-1, 1\}$ such that  
$$ 
\left \| \mathbf y_k  \cdot \mathbf w_k(\mathbf j) \right\| = \Bigl\| \eta j_h + p + \frac 12 \Bigr\| = \frac 12.
$$
For each $\mathbf v$ with $|\mathbf v| < \frac{1-2\delta}{2  |\mathbf y_k|}$, we have 
$$ |  \mathbf y_k \cdot \mathbf v| \le |\mathbf y_k | \cdot |\mathbf v| < \frac 12 - \delta.$$
Therefore,
$$\eta j_h + p + \delta < \mathbf y_k \cdot (\mathbf w_k(\mathbf j) + \mathbf v) < \eta j_h + p + 1 - \delta,$$
i.e.,
$$ \| \mathbf y_k \cdot (\mathbf w_k(\mathbf j)+ \mathbf v) \| > \delta.$$
Let $B( \mathbf w, r) = \{ \mathbf v \in \mathbb R^n:  | \mathbf w - \mathbf v | < r\}$ 
be the ball centered at $\mathbf v$ of radius $r$
and set
$$
G_{k,\delta} := \bigcup_{0 \le j_1, \dots , j_n <  
|\mathbf y_k|}  B \left( \mathbf w_k(\mathbf j), \frac{1-2\delta}{2  |\mathbf y_k|} \right). 
$$
Then $G_{k,\delta}$ is contained in $E_{k, \delta}$.
The balls composing $G_{k,\delta}$ are disjoint since 
$$
|\mathbf w_k(\mathbf j) - \mathbf w_k(\mathbf j')| \ge \frac{1}{| y_{k,h}|} 
\ge \frac{1}{| \mathbf y_k|}, \quad \hbox{for $\mathbf j \ne \mathbf j'$}. 
$$
For $\  |\mathbf j - \mathbf j' |_{\infty} \le d$, with $d$ in $\mathbb N$, we have 
$$
| \mathbf w_k(\mathbf j)  - \mathbf w_k(\mathbf j') | \le \frac{\sqrt{nd^2 + 2d +1}}{ | y_{k,h}|} 
\le \frac{\sqrt{nd^2 + 2d +1}}{|\mathbf y_k| / \sqrt n}  \le \frac{n(d+1)}{|\mathbf y_k|} .
$$ 
Therefore, any ball of radius bigger than $\frac{n (d+ 1)}{2 |\mathbf y_k|} + \frac{1-2\delta}{2  |\mathbf y_k|} $ 
contains $d^n$ balls of the form $B \left( \mathbf w_k(\mathbf j), \frac{1-2\delta}{2  |\mathbf y_k|} \right)$.
Hence, each ball of $G_{k-1,\delta}$ contains at least    
$$
\left( \frac{1-2\delta}{n} \left( \frac{|\mathbf y_{k}|}{|\mathbf y_{k-1}|} -1 \right) -2 \right)^n
$$ 
balls of $G_{k,\delta}$. 
The condition $|\bfy_{k+1}| / |\bfy_k| \ge 4n /(1-2\delta) +1$ implies that 
$$
m_k \ge \left( \frac{1-2\delta}{n} \left( \frac{|\mathbf y_{k}|}{|\mathbf y_{k-1}|} -1 \right) -2  \right)^n \ge  
\left( \frac{2(1-2\delta)}{1-2\delta +4n} \cdot \frac{ |\bfy_{k}| }{|\bfy_{k-1}|} \right)^n,
$$
where $m_k$ is a lower bound for the number of balls of level $k$ contained 
in a ball of level i$k-1$. 
Any two balls are separated by at least $\ep_k := 2\delta/|\mathbf y_k|$. 
Putting $C =  \left( \frac{2(1-2\delta)}{1-2\delta + 4 n} \right)^n$,
the mass distribution principle implies that 
$$
\dim_H \cS_\delta \ge \dim_H \left( \bigcap_k G_{k,\delta} \right) 
= \liminf_k \frac{ \log (m_1 \cdots m_{k-1})}{ - \log m_k \ep_k^n}        
= \liminf_k \frac{ n \log |\mathbf y_{k-1}| + k C }{ \log |\mathbf y_{k-1}| } = n.
$$
This establishes the theorem. 
\end{proof}

\medskip

\begin{proof}[Proof of Theorem~\ref{thm:1.6}]     
We keep the notation from Subsection 1.2 and Theorem~\ref{thm:1.6}. 
In particular, 
$$
\uy_k = {}^t(y_{k,1}, \dots , y_{k,n}) , \quad (k \ge 1),
$$
is a sequence of best approximation associated to the matrix $A$
and we set $Y_k := | \uy_k |$ for $k \ge 1$. 
We assume that the quotient $Y_{k+1} / Y_k$ tends to infinity with $k$. 
Let $\delta$ be in $(0, 1/2)$. 
Let $\ux$ be in ${\cS}_{\delta}$, that is, such that
\begin{equation}\label{eqn:5.1}
\Vert y_{k, 1} x_1 + \ldots + y_{k, n} x_n
\Vert
\ge {\delta}, \quad \hbox{for all $k \ge 1$}.   
\end{equation} 
Let  $\uq$ be a non-zero  integer $m$-tuple and let  $k$ be the index
defined
by the inequalities
$$
Y_{k} \le (2 m \delta^{-1})^{m/n} \, \vert \uq \vert^{m/n} <
Y_{k+1}.
$$
Taking into account that $M(\uy_k) \le Y_{k+1}^{-n/m}$, the inequality \eqref{eqn:5.1} and 
$$
\Vert y_1 x_1 + \dots + y_n x_n \Vert
\le n\vert \uy \vert \max_{1\le i\le n}\Vert L_i(\uq) - x_i \Vert
+ m \vert \uq \vert M(\uy)
$$
applied for $\uy = \uy_{k}$ give
$$
\delta   \le n \, Y_k \, \Vert A\uq - \ux \Vert +
m \vert \uq \vert \, Y_{k+1}^{-n/m}, 
$$
thus,
$$
\delta   \le n (2 m \delta^{-1})^{m/n} \, \vert \uq \vert^{m/n}
\, \Vert A\uq - \ux \Vert +
m  (2 m \delta^{-1})^{-1}.
$$
Consequently, we get
$$
\Vert A\uq - \ux \Vert \ge {\delta  \over 2 n (2 m \delta^{-1})^{m/n}} \,
\vert \uq \vert^{-m/n}. 
$$
By letting $\delta$ tend to $1/2$, this completes the proof of the second assertion of Theorem~\ref{thm:1.6}. 

For the first assertion, under the assumption that $Y_k^{1/k}$ tends to infinity, we proceed 
as in the proof of Theorem \ref{thm:1.3} to extract a subsequence of $(\uy_{\varphi(k)})_{k \ge 1}$ 
of $(\uy_k)_{k \ge 1}$ with the property that 
$$
Y_{\varphi(k)} \ge R Y_{\varphi(k-1)},   \quad
Y_{\varphi(k-1) + 1} \ge Y_{\varphi(k)} / R, 
\quad \hbox{for $k \ge 2$,}
$$
where $R = 4n/ (1 - 2 \delta) + 1$ is given by Theorem \ref{thm:5.1}. Then, everything goes 
exactly as above. We omit the details.
\end{proof}

\section*{Acknowledgement}   
DK is supported by the NRF of Korea (NRF-2015R1A2A2A01007090).
SL is supported by Samsung Science and Technology Foundation under Project No. SSTF-BA1601-03.
MR is supported by National Science Centre grant 2014/13/B/ST1/01033 (Poland).


\end{document}